\documentclass[10pt]{amsart}

\usepackage{amsmath, amsfonts, amssymb, amscd, graphicx,
ktkmath-eng, hyperref} 

\begin{document}

\newcommand{\pp}{\mathbf p}
\newcommand{\zz}{\mathbf z}

\title{Topological invariants and Holomorphic Mappings}
\author{R. E. Greene, K.-T. Kim \and N. V. Shcherbina}

\address{(Greene) Department of Mathematics, University of California, Los Angeles, CA
90095 U. S. A.}
\email{greene@math.ucla.edu}

\address{(Kim) Department of Mathematics, Pohang University of Science and Technology, Pohang City 37673 South Korea.}\
\email{kimkt@postech.ac.kr}

\address{(Shcherbina) Department of Mathematics, University of Wuppertal, Wuppertal 42119 Germany}
\email{shcherbina@math.uni-wuppertal.de}


\maketitle


\begin{quote} \small
\textsc{Abstract.} 
Invariants for Riemann surfaces covered by the disc and for hyperbolic
manifolds in general involving minimizing the measure of the image over 
the homotopy and homology classes of closed curves and maps of the 
$k$-sphere into the manifold are investigated.  The invariants are 
monotonic under holomorphic mappings and strictly monotonic under 
certain circumstances. Applications to holomorphic maps of annular 
regions in $\mathbb{C}$ and tubular neighborhoods of compact totally 
real submanifolds in general in $\mathbb{C}^n$, $n \geq 2$, are given. 
The contractibility of a hyperbolic domain with contracting 
holomorphic mapping is explained.
\end{quote}

\tableofcontents

\section{Introduction}

It is an elegant and effective technique in Riemannian geometry to consider
the minimum-length curve in each nontrivial free homotopy class of closed 
curves in a compact Riemannian manifold.  Such a minimum-length curve 
always exists and is a smooth closed geodesic.  This idea is, for example, the
basic step in proving that a compact, orientable manifold of even dimension
with everywhere positive sectional curvature is simply connected (cf.\ 
\cite{Petersen}  p.\ 172, Theorem 26). 
If the Riemannian manifold is not compact, then a minimal length curve may 
not exist, but it remains of interest to consider the infimum of the lengths of
closed curves in the free homotopy class and also the infimum over all 
nontrivial free homotopy classes as well.

The purpose of this paper is to examine this general idea in the context of 
complex manifolds and also to consider the corresponding possibilities for 
the situation of maps of the $k$-sphere which are homotopically nontrivial, 
with length replaced by a suitable idea of $k$-dimensional measure.

For Riemann surfaces which are covered by the unit disc $\Delta$---that is, 
all Riemann surfaces except $\CC, \CC\setminus \{0\}, \CC \cup \{\infty\}$, 
and compact surfaces of genus 1---the situation becomes one 
of Riemannian metrics: If $\pi\colon \Delta \to M$ is a holomorphic 
covering map of the Riemann surface, then the Poincar{\' e} metric 
on the unit disc $\Delta$ ``pushes down'' to a smooth Riemannian (Hermitian) 
metric on $M$, i.e., there is a unique metric on $M$ such that $\pi$ is a local 
isometry.  This canonical push-down metric on such $M$'s has the property 
that, if $F\colon M_1 \to M_2$ is a holomorphic map, then 
$F$ is metric nonincreasing.  This property will
be exploited early in the paper to recover various classical results on maps of
annular regions in $\CC$ from the viewpoint of minimization of lengths of 
curves in free homotopy classes.

Extension of these ideas to higher dimensional complex manifolds and to 
$k$-homotopy classes, $k>1$, involves new features: The natural metric 
to consider is of course the ``hyperbolic metric'' introduced by S. Kobayashi, 
which defines a natural extension to all dimensions of the canonical 
metric construction for Riemann surfaces just discussed.  
But a new aspect arises: the Kobayashi metric
on a hyperbolic manifold (in the sense of Kobayashi) need not be a Riemannian
metric in dimension $\ge 2$.  It is, however, a Finsler metric 
with an infinitesimal form known as the Kobayashi-Royden metric 
\cite{Royden70}. But for such an infinitesimal metric, which is only upper 
semi-continuous and does not satisfy the triangle inequality in general, 
it is necessary to exercise care in defining
$k$-dimensional volume measure to the image of a $k$-sphere in the 
manifold.  We shall consider primarily the details of this matter only for some 
subsets of $\CC^n$, $n\ge 2$ (cf. Sections \ref{mono2} and 
\ref{general-tubular}), although in principle greater generality would 
be possible with other hypotheses.  This relevant
general idea is $k$-dimensional Hausdorff measures.  These ideas are applied
in the final sections to obtain results for holomorphic mappings of tubular 
neighborhoods of totally real $k$-spheres in $\CC^n$, analogous to the earlier 
results on maps of annular regions (cf.\ Section \ref{RiemSurf1}).  
It is also natural to consider the tubular neighborhoods of more 
general compact totally real submanifolds.  For such a general case, 
the degree of maps, a homology invariant is more appropriate and has 
been investigated.  

\bigskip

\begin{center}
\textbf{Some Notation}
\end{center}
\bigskip

\begin{itemize}
\item $F^\textrm{Kob}_U$ : the Kobayashi-Royden metric of an open
set $U$ in $\CC^n$ \\
\item $\mu^k_{X,d} (A)$ : the $k$-dimensional Hausdorff measure of the 
subset $A$ in the metric space $(X,d)$ \\
\item $\mu^k_\textrm{Euc} (B)$  : the $k$-dimensional Hausdorff measure 
of the subset $B$ in $\CC^n$ with respect to the Euclidean norm \\
\item $\mu^k_{\textrm{Kob}, U} (C)$ : the $k$-dimensional Hausdorff measure 
of the subset $C$ of an open set $U$ in $\CC^n$ with respect to the Kobayashi 
distance
\end{itemize}
\medskip

\section{The $\ell_1$-invariant}

Let $(X, \rho)$ be a metric space with the metric $\rho$. 
For a continuous curve $\alpha \colon [a,b] \to X$, and a
partition of the interval $[a,b]$ 
\[
P := \{ t_k \colon k=0, 1, \cdots, N, \textrm{ with } 
a=t_0 < t_1 < \cdots < t_N = b\}
\]
for some positive integer $N$, let
\[
s(\alpha, P) := \sum_{k=1}^N \rho(\alpha(t_{k-1}), \alpha(t_k)).
\]
If the set 
$\{s(\alpha, P)\colon P \textrm{ a partition of } [a,b]\}$ 
is bounded above, we say that $\alpha$ is \textit{rectifiable} and define 
the \textit{length} $L(\alpha)$ of $\alpha$ by 
\[
L(\alpha) = \sup \{s(\alpha, P) \colon P \textrm{ a partition of } [a,b]\}.
\]
As usual, reparametrizations of rectifiable curves are rectifiable and the 
length is independent of parametrization.

\begin{definition}
\textup{
A map $F\colon X_1 \to X_2$ from a metric space $(X_1, \rho_1)$ to
$(X_2, \rho_2)$ is called \textit{distance nonincreasing} if 
$\rho_2 (F(x), F(y)) \le \rho_1 (x, y)$ for all $x, y \in X_1$.
}
\end{definition}

Notice that, if a map $F$ is distance nonincreasing, then 
the $F$-images of rectifiable curves are rectifiable and the 
$F$-images have length $\le$ the length of the original curve, 
i.e., length $F(\gamma(t)) \le$ length $\gamma(t)$,
for all rectifiable curves $\gamma(t)$ in $X_1$.

\medskip
\begin{definition}[The $\ell_1$-invariant]
\textup{
Consider a metric space $(X,\rho)$ which is not simply connected.  Then 
the $\ell_1$-\textit{invariant} of the metric space $X$ is defined to be
\[
\ell_1 (X) := \inf \{L(\alpha) \colon \alpha ~\textrm{a
non-contractible rectifiable loop in } X\}, ~\inf\varnothing = +\infty.
\]
}
\end{definition}
\medskip

\begin{proposition} \label{ell-inv}
If $f\colon (X_1, \rho_1) \to (X_2, \rho_2)$ is a distance nonincreasing map, 
and if the induced homomorphism 
$f_* \colon \pi_1 (X_1) \to \pi_1 (X_2)$ is injective, then 
$\ell_1 (X_2) \le \ell_1 (X_1)$.
\end{proposition}

\begin{proof}
For a loop $\alpha$ in $X_1$ denote by $[\alpha]$ the set 
of all loops that are free homotopic to $\alpha$  in $X_1$.  Set
\[
\ell_1([\alpha]) = \inf \{L(\beta) \colon \beta \in [\alpha]\}.
\]
Then 
\[
\ell_1 (X_1) = \inf \{\ell_1([\alpha]) \colon \alpha 
\textrm{ non-contractible in } X_1\}.
\]
Let $\epsilon > 0$.  Take a noncontractible loop $\beta$ in $X_1$ with
$\ell_1 ([\beta]) < \ell_1(X_1) + \epsilon$.  Then $f \circ \beta$ is 
noncontractible in $X_2$, since $f_*$ is injective.  Then 
$\ell_1([f\circ \beta]) \ge \ell_1 (X_2)$.  The distance nonincreasing property
of $f$ implies that 
\[
\ell_1 (X_2) \le \ell_1([f\circ \beta]) \le \ell_1([\beta]) < \ell_1(X_1) + \epsilon.
\]
Since this holds for any $\epsilon > 0$, the assertion follows immediately.
\end{proof}
\medskip

Distance nonincreasing/length nonincreasing maps arise naturally in
complex analysis.  The classical Schwarz Lemma is equivalent to the fact 
that a holomorphic function $f$ from the unit disc
$\Delta =\{z \in \CC \colon |z|<1\}$ into itself with $f(0)=0$ has the 
property that $d(0, f(z)) \le d(0,z)$, where $d=$ the Poincar{\' e} distance.  
Since the action by holomorphic isometries of $\Delta$ to itself relative 
to the Poincar{\' e} metric are transitive on $\Delta$, what is known as the 
Schwarz-Pick Lemma follows immediately:  
\textit{
If 
$f\colon \Delta \to \Delta$ is holomorphic, then 
$d(f(z_1), f(z_2)) \le d(z_1, z_2)$ for all $z_1, z_2 \in \Delta$, where
$d$ is the Poincar{\' e} distance.}

This can be extended to Riemann surfaces as follows:
If $M$ is a Riemann surface that is holomorphically covered by the 
unit disc $\Delta$. say $\pi\colon \Delta \to M$ is a holomorphic
covering map, then the covering transformation of the covering
$\pi$ are holomorphic isometries for the Poincar{\' e} metric.  It follows
that $M$ has a unique Riemannian (indeed Hermitian) metric for which 
$\pi$ is locally isometric.  Let us call this the \textit{canonical metric} 
for $M$.

If $M_1$ and $M_2$ are two such Riemann surfaces, with canonical 
(Riemannian) metrics $g_1$ and $g_2$ respectively, and if 
$F\colon M_1 \to M_2$ is a holomorphic map, then the pull back $F^* g_2$ 
to $M_1$ of the metric $g_2$, is less than or equal to $g_1$, i.e., 
\[
F^* g_2|_p \le g_1|_p
\]
for each $p \in M_1$, where 
\[
F^* g_2 |_p (v, w) = g_2|_{F(p)} (dF_p (v), dF_p (w)),
\]
for all $v, w$ in the real tangent space of $M_1$ at $p$, so $dF_p (v)$ and
$dF_p (w)$ are in the real tangent space of $M_2$ at $F(p)$.  

To see this distance nonincreasing property, note that if 
$\pi_1 \colon \Delta \to M_1$ and $\pi_2 \colon \Delta \to M_2$ are
holomorphic covering maps, then $F\colon M_1 \to M_2$ can be 
lifted to a holomorphic map $\hat F \colon \Delta \to \Delta$ in such
a way that the diagram
\[
\begin{CD} 
\Delta @>{\hat F}>> \Delta \\
@V{\pi_1}VV      @VV{\pi_2}V \\
M_1 @>>F>  M_2 
\end{CD}
\]
commutes. The map $\hat F$ is nonincreasing for the Poincar{\' e} metric
by the Schwarz-Pick Lemma. (cf. \cite{KimLee})

It follows, since $\pi_1$ and $\pi_2$ are (local) isometries, in the sense 
indicated of the inequality on the pull-back metric, that 
$F\colon M_1 \to M_2$ is length nonincreasing/distance nonincreasing 
for the canonical metrics of $M_1$ and of $M_2$, respectively.

\section{Holomorphic maps of Riemann surfaces of general type
}
\label{RiemSurf1}

In the case that $M = \{z \in \CC \colon \frac1{\sqrt{R}} < |z| < \sqrt{R} \}$,
$R>1$, which is covered by the unit disc $\Delta$, the canonical metric on 
$M$ is straightforward to compute.  
The function $F(z) = e^{iz}$ defines a covering map of the ``strip''
$\{x+iy \colon - \ln R < y < \ln R\}$ onto $M$, while the strip is biholomorphic
to the upper half plane via the composition of a linear map, exponentiation,
and a linear fractional transformation.  Tracing through gives the formula
\[
\frac{\big(\pi / (2 \ln R)^2\big)}{r^2 \cos^2 \big(\pi \ln r / (2\ln R)\big)} dr^2 +
\frac{\big(\pi/(2 \ln R)\big)^2}{\cos^2 \big(\pi \ln r / (2\ln R)\big)} d\theta^2
\]
for the canonical metric on $M$ in $(r,\theta)$ polar coordinates. (Cf.\
\cite{GKK}, p.\ 39.)

The free homotopy classes of closed curves in $M$ are characterized by
winding number (around $0$) and are nontrivial for all winding numbers 
except $0$.  Since winding number is $\frac1{2\pi}$ times the total
change in polar angle $\theta$ around the curve, it follows easily
that, for any rectifiable curve $\gamma$ in $M$, the length of the curve
in the canonical metric is $\ge \pi^2/\ln R$.  Thus the $\ell_1$-invariant
of $M$ in its canonical metric is $\ge \pi^2/\ln R$ and indeed
\begin{equation}
\ell_1 (M) = \frac{\pi^2}{\ln R},  \label{M-for-Ann}
\end{equation}
with the infimum realized by once-around the curve 
$r=1$ ($\theta$ goes from $0$ to $2\pi$), either clockwise or 
counterclockwise.  (Throughout, we are using the Poincar{\' e} metric on 
the unit disc $\Delta$ to be $4 (1-z\bar z)^{-2} dz \ d\bar z $ so that
the curvature $\equiv -1$.)
\medskip

Since $\{z \in \CC \colon A < |z| < B\}$ is linearly biholomorphic to 
$\{ z \colon \sqrt{A/B} < |z| < \sqrt{B/A} \}$, the $\ell_1$-invariant of 
$\{z \in \CC \colon A < |z| < B\}$ is equal to $\pi^2/\ln (B/A)$. The 
$\ell_1$-invariant is preserved by biholomorphic mapping, so the
classical result follows:

\begin{theorem}[Hadamard] \label{AnnThm0}
The region $\{z \colon A_1 < |z| < B_1\}$ is biholomorphic to 
$\{z \colon A_2 < |z| < B_2\}$ if, and only if, $B_1/A_1 = B_2/A_2$.
\end{theorem}

This result, originally proved by Hadamard (Cf. \cite{Littlewood}), 
is usually proved by non-metric methods, e.g., Schwarz Reflection, 
\cite{GK, Ahlfors, Rudin} et al.
\medskip

The nonincreasing property of the $\ell_1$-invariant gives an extension of this
result:

\begin{theorem} \label{AnnThm}
If $0<A_1<B_1$, $0<A_2<B_2$ and $B_1/A_1 > B_2/A_2$, then 
every holomorphic mapping $f\colon \{z \colon A_1 < |z| < B_1\}
\to \{z \colon A_2 < |z| < B_2\}$ is homotopically trivial, that is, it is 
homotopic to a constant map.
\end{theorem}

\noindent
\textit{Proof}. The map $f$ is homotopically trivial if and only if the
$f$-image of the curve $\gamma (t) = \frac12(A_1 + B_1) (\cos t, \sin t),
~(t \in [0, 2\pi])$ is homotopic to a constant curve: this is an immediate
consequence of covering space theory.  If $f$ is not homotopically
trivial, then the free homotopy class of the $f$-image indicated is 
nontrivial.  Then Proposition \ref{ell-inv} gives that 
\[
\ell_1 (\{z \colon A_2 < |z| < B_2\}) \le \ell_1 (\{z \colon A_1 < |z| < B_1\}),
\]
and the result follows from the formula \eqref{M-for-Ann}
for the $\ell_1$-invariant.
\hfill
$\Box$
\bigskip

This Theorem \ref{AnnThm} implies in particular the historical result: 

\begin{theorem}[de Possel \cite{dePossel}] 
\label{H-2}
There is a 1-1 holomorphic mapping $g\colon \{z \colon 0<A_1 < |z| < B_1\}
\to \{z \colon 0< A_2 < |z| < B_2\}$ whose image separates the boundary 
components of $\{z \colon A_2 < |z| < B_2\}$ if and only if
$B_1/A_1 \le B_2/A_2$.
\end{theorem}
\medskip

Notice that the condition on $g$ implies that $g$ is injective on homotopy; 
For instance the $g$-image of the curve $\gamma$ in the preceding proof 
will be a closed curve in $\{z \colon A_2 < |z| < B_2\}$ that goes around 
the origin exactly once.  Notice also that the ``if'' part of this theorem is 
obviously true by a complex linear map.  (Cf.\ \cite{Goluzin, Grotzsch} for the 
original proof.)
\bigskip

A conclusion by the same argument holds for multiply connected domains and
Riemann surfaces as well.

\begin{theorem} 
Let $X$ and $Y$ be Riemann surfaces holomorphically covered by the unit disc 
with $\ell_1 (X) < \ell_1 (Y)$.  If $f\colon X \to Y$ is a holomorphic map, then
it cannot be injective on homotopy.  In particular, $f$ cannot be a homotopy 
equivalence.
\end{theorem}

\section{Non-Riemannian Kobayashi hyperbolic case}

The only properties of the ``canonical metric'' on Riemann surfaces
covered by the disc that were crucial here were that holomorphic maps
were distance nonincreasing and that the canonical metric was locally
comparable (in both directions) with any metric derived from local
coordinates.  In this context, it is clear that the whole viewpoint has an 
immediate extension to complex manifolds which are ``hyperbolic'' in
the sense introduced by S. Kobayashi.  In this section, we consider only
complex manifolds (or complex spaces) that are hyperbolic in the sense
of Kobayashi, namely, those for which the Kobayashi pseudodistance is
an actual distance function.
\cite{Kobayashi70, Kobayashi98}.

In this setting, one can define again the $\ell_1$-invariant $\ell_1(M)$ 
of a hyperbolic 
manifold to be the infimum of the lengths of closed curves that are not 
freely homotopic to 0 (assuming $M$ is not simply connected). Then following
the pattern of before one gets the results:

\begin{theorem} 
If $f\colon M_1 \to M_2$ is a holomorphic map of hyperbolic manifolds, and if
$f$ is injective on free homotopy classes of closed curves (in the sense that if 
$\gamma(t)$ is a closed curve in $M_1$ not freely homotopic to a constant,
then $f(\gamma(t))$ is not freely homotopic to a constant), then
\[
\ell_1 (M_2) \le \ell_1 (M_1).
\]
\end{theorem}

Thus the results discussed in the concrete instances of Riemann surfaces 
covered by the unit disc and of annular regions in $\CC$ in particular, 
can be extended to
far more general settings. The idea of the $\ell_1$-invariant via free homotopy
classes of closed curves can also be extended to higher dimensional homotopy
classes of maps of the $k$-sphere into complex hyperbolic manifolds, and to some
extent, into general length spaces.  These methods will be explored in subsequent 
sections.

\section{Holomorphic mappings of tubular domains}

The previous discussion of annular regions in $\CC$ has a straightforward
extension to \textit{tubular domains} in $\CC^n$, $n > 1$. For this, define, for $0 < r< 1$,
\[
T^n (r) = \{ \vec z = (z_1, \ldots, z_n) \in \CC^n \colon
|z_1|=1, \|\vec z - (z_1, 0, \ldots, 0)\| < r \},
\] 
where  $\|\cdot\|$ is the Euclidean norm.
Set $A(r_1,r_2) = \{ z \in \CC \colon r_1 < |z| < r_2 \}$ where
$0<r_1<r_2$.  The projection map $P_n (z_1, \ldots, z_n) = z_1$ takes
$T^n (r)$ onto $A(1-r, 1+r)$ and there is a natural injection 
$J_n (z) = (z,0,\ldots, 0)$ mapping $A(1-r,1+r)$ into $T^n (r)$.  These maps are
homotopy equivalences: $J_n$ and $P_n$ are homotopy inverses to each
other.  This observation together with Theorem \ref{AnnThm} gives rise to
a comparison result on these \textit{tubular domains} (meaning tubular 
neighborhoods of a circle, where the dimensions need not be equal):

\begin{theorem} 
If $f\colon T^n (r) \to T^m (s)$ with $r, s$ between $0$ and $1$, is
holomorphic and if $s<r$, then $f$ is homotopic to a constant map.
\end{theorem}

\noindent
\textit{Proof}. The map $f$ is homotopic to a constant map if and only if the 
$f$-image of the closed curve
$\Gamma(t) = (e^{it}, 0, \ldots, 0), ~t \in [0,2\pi]$  in $T^n (r)$ is homotopic
to a constant curve in $T^m (s)$. This is again by the standard covering space
theory.  But, this happens if and only if the holomorphic map 
$F\colon A(1-r,1+r) \to A(1-s,1+s)$  defined by $F := P_m \circ f \circ J_n$
is homotopic to a constant.  This last follows from the homotopy equivalences
already noted.  Now Theorem \ref{AnnThm} implies that, if $s<r$, this map
$F$ is, and consequently $f$ is also, homotopic to a constant.  
\hfill
$\Box$
\medskip

\section{Higher homotopy invariants} 
\label{Higher}

It is natural to consider extending the previous introduced ideas about 
curve lengths in free homotopy classes of closed curves to higher dimensional
homotopy classes. In particular, if $(X, d)$ is a metric space (which we
assume arc-wise connected for convenience), then consider
continuous maps $f\colon S^k \to X$, $k>1$, 
$S^k = \{ (x_0, \cdots, x_k) \in \RR^{k+1} \colon x_0^2 + \ldots + x_k^2 = 1\}$
and define two such $f_0, f_1 \colon S^k \to X$ to be free-homotopic if
there is a continuous function $H\colon S^k \times [0,1] \to X$ such that,
for all $p$, $H (p,0) = f_0 (p)$ and $H (p,1) = f_1 (p)$.
If one has a situation where, for suitably restricted $f$, one can define the 
measure of (the image of) such $f\colon S^k \to X$, then one could imitate
the definition of the $\ell_1$-invariant corresponding to the $k=1$ case.  But
certain difficulties arise: One needs an idea of such a measure and one would
want to know that each free homotopy class of continuous maps
$S^k \to X$ would contain at least one map $f$ with the associated measure
of the image of $f$ being finite.  This would correspond to rectifiable 
curves and their lengths, and free-homotopy classes of curves, as in
earlier sections.

In the case where $(X, d)$ is a Riemannian manifold and $d$ its Riemannian
metric distance, the right idea of measure is, for smooth maps 
$f\colon S^k \to X$, the measure of $f|_A$, $A \subset S^k$, $A=$ the set of
points of $S^k$ where the differential $df$ has rank $k$, and the measure
of $f|_A$ is the usual Riemann-metric induced measure on $k$-dimensional
submanifolds. This measure is 
\[
\mu(f) = \int_{S^k} |\textsl{volume form}|
\]
where the $|\textsl{volume form}|$ on $(v_1, \ldots, v_k)$, where 
$v_1, \ldots, v_k \in T_p S^k$, is equal to the absolute value of 
the $k$-dimensional Riemannian volume at $f(p)$ of 
$f_* (v_1)\wedge \cdots \wedge f_* (v_k)$.  
It is easy to show (and well-known) that every free homotopy class of 
continuous maps $S^k \to X$ in this case ($X$ is a Riemannian manifold)
contains a smooth (not necessarily everywhere nonsingular) map:
If $f$ is a continuous map in the free homotopy class then a sufficiently good
smooth approximation of $f$ will be in the same free homotopy class 
(by deformation along minimal geodesics).  
So every free homotopy class contains
finite $k$-measure maps and hence one can define the infimum of measures
over finite-measure (smooth, e.g.) maps in the class and also an $\ell_k$ 
invariant (= infimum over all nontrivial free homotopy classes).
\medskip

In the case of more general metric spaces, it is not in general clear that 
any useful concept of $k$-dimensional measure exists. 
But in the case of metrics which are locally equivalent to Riemannian 
metrics, the familiar general idea of Hausdorff $k$-dimensional measure, 
can be used.  

\begin{definition}[Cf. \cite{Kobayashi98}, p.\ 343]
\textup{
Given a metric space $(X, d)$, for a nonnegative real number $k$ and a
subset $A$, the \textit{$k$-dimensional Hausdorff measure} 
$\mu^k_{X, d} (A)$ is defined as
\[
\mu^k_{X,d} (A) = \sup_{\epsilon > 0} \inf \Big\{ \sum_{i=1}^\infty \big(\delta (A_i)\big)^k \colon A = \bigcup_{i=1}^\infty A_i, \delta (A_i)<\epsilon\Big\},
\]
where $\delta (A_i)$ is the diameter of $A_i$ defined by
\[
\delta (A_i) = \sup_{p, q \in A_i} d(p, q).
\]
}
\end{definition}
\medskip

For a smooth map $f\colon S^k \to M$ of the $k$-sphere $S^k$ into a 
Riemannian manifold $M$, this $\mu^k_M (f(S^k))$ gives essentially the same 
concept of the $k$-dimensional measure as the Riemannian measure 
already defined.  But this new notion of measure has the advantage that smoothness is not involved.
\medskip

Rather than exploring these matters further in generality, we restrict our 
attention now to the specific situation of the Kobayashi metric for a bounded 
connected open set $U$ in $\CC^n$.  The open set $U$ has its Kobayashi 
metric in the infinitesimal form, the Kobayashi-Royden metric, as it is 
usually defined by 
\[
F^\textrm{Kob}_U (p,v) = \inf \Big\{r>0 \colon f \in \cO(\Delta, U), f(0)=p, df(0) = \frac{v}r \Big\}
\]
where $\Delta = \{z \in \CC\colon |z|<1\}$ and where $\cO(\Delta, U)$ 
denotes the set of holomorphic maps from $\Delta$ to $U$.  As is 
well-known \cite{Royden71, Royden70}, the Kobayashi distance from 
$p$ to $q$ for $p,q \in U$ is
the infimum of the length of the $C^1$ curves from $p$ to $q$ in $U$ where
the length is taken to be the integral $\int_\gamma F^{Kob}_U (z, dz)$.  
(The existence of this integral was also shown in [\textit{Op.\ cit.}]).  
Moreover, it was shown in \cite{Barth} that $F^{Kob}_U (p,v)$ is locally 
comparable to the standard Euclidean norm of $v$ in $\CC^n$. This 
leads easily to the fact that for 
every free homotopy class of maps $S^k \to U$, there is a smooth map 
$f\colon S^k \to U$ (as already noted) and the smooth map will have finite 
Hausdorff $k$-measure relative to the Kobayashi metric, denoted by 
$\mu^k_{\textrm{Kob},U} (f(S^k))$.  We call such a map \textit{$k$-rectifiable} 
(meaning: Hausdorff $k$-measure is finite), and the $k$-measure just 
constructed will be called, in this paper, the \textit{Hausdorff-Kobayashi 
$k$-measure}.
\medskip

Thus, one can define an invariant
\begin{align*}
\ell_k (U) = 
&\textrm{ infimum of the Hausdorff-Kobayashi $k$-measures}\\
&\textrm{ over all the Hausdorff-Kobayashi $k$-rectifiable }\\
&\textrm{ representatives of the nontrivial free homotopy }\\
&\textrm{ classes in $U$}
\end{align*}

\begin{lemma}
If $F\colon U \to V$ is a holomorphic mapping from a bounded domain in
$\CC^n$ to a bounded domain in $\CC^m$ and if it is 
injective on free homotopy in the sense that the $F$-image of 
every nontrivial $k$-class in $U$ is nontrivial in $V$, then 
$\ell_k (V) \le \ell_k (U)$.
\end{lemma}

\noindent
\textit{Proof}.
This follows immediately from the fact that $F$ is distance nonincreasing and 
hence diameter nonincreasing with respect to the Kobayashi metrics of $U$ and
$V$, so the Hausdorff sums are not increased by composition with $F$.
\hfill $\Box$
\medskip

\section{Illustration of the $k$-homotopy invariant}
\label{Illust}

The most natural way to find a bounded connected open set in $\CC^n$ which 
has a nontrivial $k$-homotopy class is to take a tubular neighborhood of an
embedded $k$-sphere.  This is particularly interesting from the viewpoint of 
complex analysis if one takes the $k$-sphere to be totally real as a submanifold.
In particular, the unit $k$-sphere in $\RR^{k+1}$ is totally real in $\CC^{k+1}$.
Let $T_r$ be the radius $r>0$ tubular neighborhood 
\begin{align*}
T_r = \{ (x_0, \ldots, x_k) + \lambda \vec u \colon
& x_0^2 + \ldots + x_k^2 = 1, \\
& \vec u \in \CC^{k+1}, ~\|u\|=1, ~0 \le\lambda < r \}.
\end{align*}
For small values of $r$, this is a strongly pseudoconvex domain with $C^\infty$
boundary (actually $0<r<\frac12$ suffices).  According to \cite{Graham}, the
Kobayashi-Royden metric goes to $+\infty$ near the boundary 
relative to the Euclidean metric.
Note that there is a natural smooth projection, say $P\colon T_r \to S^k$,
defined by $P(z) =$ the closest point to $z$ in the Euclidean sense in the set 
$S^k$.  The injection $S_k \hookrightarrow T_r$ is a homotopy equivalence 
with $P$ its
homotopy inverse.  And the identity (injection) map $S_k \to T_r$ is
homotopically nontrivial.  There may be other smooth maps, say $\Gamma
\colon S^k \to T_r$, homotopic to the injection which have the associated
Hausdorff-Kobayashi $k$-measure assigned to them being less than the 
measure assigned to the injection.  But the infimum will be realized among
the family of maps with image lying in some set of the form
\[
\{\vec x + \lambda \vec u \colon \vec x \in S_k, \vec u \in \CC^{k+1}, 
~\|u\| = 1, ~0 \le \lambda \le r_1\}
\]
for some $r_1 < r$: this follows easily from the estimates on the growth of the 
Kobayashi metric near the boundary of $T_r$ compared to the Euclidean
metric. This will yield a conclusion similar to Theorem \ref{AnnThm0},
once a technical point about the Kobayashi/Royden infinitesimal metric is 
established.  This will be discussed further in the following sections.

\section{Monotonicity of the Kobayashi metric}
\label{mono1}

To find some analogue in this situation of Theorems \ref{AnnThm0},
one needs to have an idea of strict monotonicity of the Kobayashi distance 
for the domains (open connected subsets) in $\CC^n$.  Specifically, if 
two domains $\Omega_1$ and $\Omega_2$ satisfy $\Omega_1 \subset
\Omega_2$, then of course  their respective Kobayashi-Royden metrics 
$F^\textrm{Kob}_{\Omega_1}, F^\textrm{Kob}_{\Omega_2}$ satisfy 
$F^\textrm{Kob}_{\Omega_2} (p,v) \le F^\textrm{Kob}_{\Omega_1} (p,v)$ 
for all $p \in \Omega_1$ and $v \in \CC^n$.  
But we would like to obtain a strict comparison between them.

\begin{lemma} \label{compare}
If $\Omega_1$ and $\Omega_2$ are bounded domains in $\CC^n$ satisfying 
$\Omega_1 \subset\subset \Omega_2$ (relatively compact), then for any 
compact subset $K$ of $\Omega_1$ there exists a constant $c_K$ with 
$0<c_K<1$ such that 
\[
F^\textrm{Kob}_{\Omega_2} (p,v) < c_K F^\textrm{Kob}_{\Omega_1} (p,v),
\]
for any $p \in K$ and any $v \in \CC^n$.
\end{lemma}

\noindent
\textit{Proof}. 
Let the bounded domains $\Omega_1$ and $\Omega_2$ in
$\CC^n$ satisfy $\Omega_1 \subset\subset \Omega_2$, i.e., $\Omega_1$
is a relatively compact subdomain of $\Omega_2$. 

Then there exists $\delta >0$ such that the Euclidean distance between 
$\Omega_1$ and $\CC^n \setminus \Omega_2$ is at least $\delta$.
\medskip

Let $p \in \Omega_1$. Notice that there exist positive numbers $b$ and $B$ such that 
\[
b \le F^\textrm{Kob}_{\Omega_1} (p, v) \le B, ~\forall v \in \CC^n \textrm{ with }
\|v\| = 1,
\]
where $\|~\|$ denotes the standard Euclidean norm.
Fix $p \in \Omega_1$ and $v \in \CC^n$ with $\|v\| = 1$. 
Let $\epsilon > 0$ be given arbitrarily, and then
choose $h\in \mathcal{O} (\Delta, \Omega_1)$
satisfying $h(0)=p$, $h'(0) = v/r$ for some $r>0$ and 
\[
F^\textrm{Kob}_{\Omega_1} (p,v) \le r < F^\textrm{Kob}_{\Omega_1} (p,v) + \epsilon.
\]
Take
\[
\tilde h (z) = h(z) + \delta z v.
\]
Then 
\[
\tilde h(0) = h(0) = p, ~\tilde h'(0) = h'(0) + \delta v = \Big(\frac1r + \delta \Big) v
\]
and 
\[
\tilde h(z) \subset \Omega_2,
\ \forall z \in \Delta,
\]
since $\|\tilde h(z) - h(z)\| = \| \delta z v\| < \delta$ for every $z \in \Delta$.
So $\tilde h \in \mathcal{O} (\Delta, \Omega_2)$.  
\medskip

This implies 
\begin{align*}
F^\textrm{Kob}_{\Omega_2} (p,v) &\le \frac1{1/r + \delta } \\
& = \frac{r}{1+r\delta} \\
& \le \frac{F^\textrm{Kob}_{\Omega_1}(p,v) 
    + \epsilon}{1+\delta F^\textrm{Kob}_{\Omega_1} (p,v)} \\
& \le \frac1{1+ \delta b} (F^\textrm{Kob}_{\Omega_1}(p,v) + \epsilon).
\end{align*}
Since $\epsilon>0$ is arbitrary, we obtain that
\[
F^\textrm{Kob}_{\Omega_2} (p,v) 
\le \frac1{1+\delta b} F^\textrm{Kob}_{\Omega_1} (p,v) 
\]
for any $v \in \CC^n$, due to the homogeneity of the Kobayashi-Royden
metric.

Note that $b$ depends on the location of $p$.  But on a compact set it 
stays bounded away from zero, and the desired conclusion follows immediately.
\hfill $\Box$
\bigskip

The restriction in Lemma \ref{compare} to a compact set $K$ can be removed.

\begin{lemma} \label{compare-2}
If $U$ is a bounded open set in $\CC^n$ with its closure contained in another
bounded open set $V$, then there is a constant $c \in (0,1)$ such that 
\[
F^\textrm{Kob}_V (p,v) \le c F^\textrm{Kob}_U (p,v) \quad \forall (p, v) \in U \times \CC^n,
\]
where $F^\textrm{Kob}_U$ and $F^\textrm{Kob}_V$ are the infinitesimal Kobayashi-Royden metrics
of $U$ and $V$, respectively.
\end{lemma}

\noindent
\textit{Proof}.
Denote by $\textrm{cl}(A)$ the closure in $\CC^n$ of the subset $A$ of $\CC^n$.
There is a bounded open set $W$ satisfying
\[
\textrm{cl} (U) \subset W \subset \textrm{cl} (W) \subset V.
\]
With $W$ so chosen, we have:
\[
F^\textrm{Kob}_W (q, v) \le F^\textrm{Kob}_U (q, v) \quad \forall (q, v) \in U \times \CC^n, 
\]
and
\[
F^\textrm{Kob}_V (p,v) \le F^\textrm{Kob}_W (p,v) \quad \forall (p, v) \in W \times \CC^n.
\]
Lemma \ref{compare} gives that there is a constant $c$ with $0<c<1$ such 
that 
\[
F^\textrm{Kob}_V (q, v) \le c F^\textrm{Kob}_W (q, v) \quad \forall (q, v) \in \textrm{cl} (U) \times \CC^n.
\]
In particular, this yields that
\[
F^\textrm{Kob}_V (p,v) \le cF^\textrm{Kob}_W (p,v) \le c F^\textrm{Kob}_U (p,v) \quad \forall (p,v) \in U \times
\CC^n,
\]
as desired.
\hfill $\Box$


\section{The example from Section \ref{Illust} concluded}
\label{mono2}

The results of Section \ref{Higher} and Section \ref{Illust} together with the 
growth of the Kobayashi-Royden metric established in \cite{Graham} can be
combined to establish the results
about the domains $T_r$ defined in Section \ref{Higher}:

First we note that the $\ell_k$-invariants are nonzero in this case.  As before, 
we assume that all $r$-values are small enough that $T_r$ is strongly 
pseudoconvex.

\begin{lemma} 
With $T_r$ as defined in Section \ref{Higher}, and $\ell_k$ defined as before,
\[
\ell_k (T_r) > 0.
\]
\end{lemma}

\noindent
\textit{Proof}. Since $T_r$ has smooth strictly pseudoconvex boundary,
Graham \cite{Graham} gives that there is a compact subset $K$ of $T_r$ 
such that for each $p \in T_r \setminus K$ and all $v$, 
\[
F^\textrm{Kob}_{T_r} (p,v) \ge \|v\|,
\]
where $\|~\|$ represents the usual Euclidean norm.
For such a fixed compact set $K$, there is a constant $c>0$ such that 
\[
F^\textrm{Kob}_{T_r} (p,v) \ge c\|v\|
\]
for all $p \in K$ and all $v \in \CC^n$. 

Replace $c$ by $\min \{c, 1\}$.  It follows 
that $\ell_k (T_r) \ge c^k \ L(T_r)$, where $L(T_r)$ represents
the infimum of the Euclidean Hausdorff $k$-measure of the image of
$S^k$ in $T_r$ not homotopic to a constant. This latter infimum is
positive by elementary considerations.
\hfill $\Box$
\medskip

\begin{theorem} \label{gen-Hada}
If $0<r<s$, then $\ell_k (T_r) > \ell_k (T_s)$.  In particular, $T_r$
is not biholomorphic to $T_s$.
\end{theorem}

\noindent
\textit{Proof}.
Consider the inclusion map $T_r \hookrightarrow T_s$. By Graham
\cite{Graham}, given any $C>0$, there is a compact subset $K$ of $T_r$ such 
that $F^\textrm{Kob}_{T_r} \ge C \cdot \textrm{(Euclidean metric)}$ at every point 
$p \in T_r \setminus K$, since $F^\textrm{Kob}_{T_r}$ goes to infinity at the boundary
of $T_r$ and hence, choosing $C$ sufficiently large, $F^\textrm{Kob}_{T_r} \ge 2 F^\textrm{Kob}_{T_s}$.
This follows since $F^\textrm{Kob}_{T_s}$ is bounded by some multiple of Euclidean 
metric on $T_r$ since $T_r \subset\subset T_s$.  By Lemma 
\ref{compare} (and its proof) there is an $\epsilon$ with $0<\epsilon < 1$
such that $F^\textrm{Kob}_{T_r} \ge (1+\epsilon) F^\textrm{Kob}_{T_s}$ at every point of $K$. Then
\[
F^\textrm{Kob}_{T_r} \ge (1+\epsilon) F^\textrm{Kob}_{T_s}
\]
at every point of $T_r$.  It follows that 
$\ell_k (T_r) \ge (1+\epsilon)^k \ell_k (T_s)$ and, since $\ell_k (T_s) > 0$,
\[
\ell_k (T_r) > \ell_k (T_s).
\]
This completes the proof.
\hfill $\Box$
\medskip

The arguments used to prove Theorem \ref{AnnThm} can be extended in a 
straightforward way to prove a corresponding result for the domains
of $T_r$ type:

\begin{theorem} \label{homotopy_thm}
If $r_1 > r_2$, then every holomorphic mapping $f\colon T_{r_1}
\to T_{r_2}$ is homotopic to a constant map.
\end{theorem}

\section{Tubular neighborhoods in general}
\label{general-tubular}

The analysis of tubular neighborhoods of totally real embeddings of spheres
in the previous section can be extended to more general circumstances.  But
this extension involves what amounts to a shift from homotopy to homology:
the role of being homotopically nontrivial is taken over by having degree
not equal to $0$.

The first step is to define the relevant concept of degree:  Suppose that $M$
is a smooth ($C^2$ suffices) compact connected submanifold without boundary of 
a Euclidean space $\RR^N$ and let 
\[
T_r = \bigcup_{w \in M} \{ v \in \RR^N \colon \|v-w\|<r \},
\]
where $\| ~ \|$ is the usual Euclidean norm.  Notice that there exists $R>0$
such that $T_r$, for any $r$ with $0<r<R$, there exists a natural projection
$\pi \colon T_r \to M$ onto $M$, defined by
\[
\pi (v) = \inf_{p\in M} \|p-v\|
\]
so that $\pi (v)$ is the closest point to $v$ among points of $M$, with respect
to the Euclidean distance.  If $g \colon M \to T_r$ is a continuous map of 
$M$ into $T_r$, then it is natural to define the \textit{degree} of $g$, denoted by
$\deg g$, to be the degree of $\pi \circ g \colon M \to M$.  (If $M$ is 
orientable, this is to be the usual $\ZZ$-valued degree.  If $M$ is nonorientable,
we take degree to be $\ZZ_2$-valued.)  If $G\colon T_{r_1} \to T_{r_2}$ is a 
continuous map, with $r_1, r_2  \in (0,R)$, we define
\[
\deg G := \textrm{the degree of } G\circ i \colon M \to T_{r_2},
\]
where $i \colon M \to T_{r_1}$ is an injection.  As well known, these concepts
of degree are multiplicative: the degree of a composition is equal to the product of
the degrees.

With these definitions in sight, the analogue of Theorem \ref{homotopy_thm}
is the following:

\begin{theorem} \label{homology-thm}
Let $M$ be a smooth compact connected totally real submanifold of 
$\CC^n$.  Then there is a constant $R>0$ such that, if $0<r<s<R$ and 
if $F\colon T_s \to T_r$ is holomorphic then, $F$ has degree zero.
\end{theorem}

The conditions on $R$ here are such that 
\begin{itemize}
\item{} $T_r$ has smooth strongly pseudoconvex boundary and 
\item{} the projection map $\pi\colon T_r \to M$ is well defined (and
continuous in particular),
\end{itemize}
whenever $0<r<R$.

Note that this result is in fact an extension of Theorem \ref{homotopy_thm}
since, in the case that $M$ is a sphere, $F$ being homotopic to a constant
is equivalent to $\deg F = 0$.  But in general, of course, $\deg F = 0$ does
not imply that $F$ is homotopic to a constant.  The most obvious example 
may be the map of $S^p \times S^p$ to itself, $p\ge 1$, identity on the first 
factor and constant on the second.
\medskip

Proof of Theorem \ref{homology-thm} follows the general pattern
of the proofs of previous theorems, but requires some preparation.  
\medskip

First, we need

\begin{definition}
\textup{
Let $k= \dim M$. For $r \in (0,R)$ as above, let 
\begin{align*}
V_r^k 
:= \inf \{ \mu^k_{\textrm{Kob}, T_r} (g(M)) \mid g\colon &M \to T_r \textrm{ continuous} \\ & \textrm{and rectifiable with } \deg g \neq 0 \}.
\end{align*}
}
\end{definition}

Now we need 

\begin{lemma} \label{pos}
$V^k_r > 0$ for any $r \in (0,R)$.
\end{lemma}

\noindent
\textit{Proof}; With $g\colon M \to T_r$, $\deg g \neq 0$, the composition
$\hat g := \pi \circ g \colon M \to M$ has nonzero degree and is hence 
surjective. 

The projection $\pi\colon T_R \to M$ admits a constant $C>0$ such that 
$\|\pi (x) - \pi (y)\| \le C \|x-y\|$ with respect to the standard Euclidean 
norm. Thus 
the Euclidean Hausdorff $k$-volume of $\pi\circ g (M)$ is at least as large 
as that of $M$, bounded away from $0$.

On the other hand, for $0<r<R$, $T_r$ is bounded strongly pseudoconvex.
So, by \cite{Graham}, the Kobayashi-Royden metric 
$F^\textrm{Kob}_{T_r}$ of $T_r$ goes to infinity compared to 
the Euclidean metric as the base point approaches the boundary of $T_r$.  
Hence there is a constant $c>0$ such that 
$F^\textrm{Kob}_{T_r} (q, w) \ge c \|w\|$
for any $q \in T_r, w \in \CC^n$.  Hence we obtain that 
\[
\mu^k_{\textrm{Kob}, T_r} (g(M)) \ge c^k \mu^k_{Euc} (g(M)).
\]
Since 
\[
\mu^k_{Euc} (g(M)) \ge C^{-k} \mu^k_{Euc} (\pi (g(M))) 
\ge C^{-k} \mu^k_{Euc} (M),
\]
it follows that $\inf_g \mu_{\textrm{Kob}, T_r}^k (g(M)) \ge
(c/C)^k  \mu^k_{Euc} (M) > 0$, as desired. 
\hfill $\Box$
\medskip

\begin{lemma} \label{k-vol-com}
$V^k_s < V^k_r$ if $ 0<r<s<R$.
\end{lemma}

\noindent
\textit{Proof}: 
By Lemma \ref{compare-2}, the strict inclusion of 
the closure of $T_r$ into $T_s$ implies that there is a constant $c$ with $0<c<1$ such that 
$F^\textrm{Kob}_{T_s} (z, v)  \le  c F^\textrm{Kob}_{T_r} (z, v)$ for any $z \in T_r$ and $v \in \CC^n$.  Hence for each $g \colon M \to 
T_r$ of nonzero degree, 
\[
\mu^k_{\textrm{Kob}, T_s} (g(M)) \le c^k \mu^k_{\textrm{Kob}, T_r} (g(M)).
\]
So $V^k_s \le c^k V^k_r$.  Since $V_s > 0$ by Lemma \ref{pos}, it follows that $V^k_s < V^k_r$.
\hfill $\Box$.
\bigskip

\noindent
\textbf{Proof of Theorem \ref{homology-thm}}. Let $0<r<s<R$ as in the 
hypothesis.  Suppose that $F \colon T_s \to T_r$ is holomorphic.  And suppose,
for a proof by contradiction that $\deg F \neq 0$.  If a continuous map $g\colon
M \to T_s$ has a nonzero degree, then 
\[
\deg (F \circ g) = (\deg F)(\deg g) \neq 0.
\]
Thus 
\[
V^k_r \le \mu^k_{\textit{Kob}, T_r} (F \circ g (M)),
\]
since $F\circ g\colon M \to T_r$ is a continuous map with nonzero degree.
But $\mu^k_{\textit{Kob}, T_r} (F\circ g(M)) \le \mu^k_{\textit{Kob}, T_s} (g(M))$,
since $F$ is Kobayashi-metric nonincreasing.  Thus 
\[
V^k_r \le \mu^k_{\textit{Kob}, T_s} (g(M)),
\]
which implies that 
\[
V^k_r \le V^k_s.
\]
This contradicts Lemma \ref{k-vol-com}, and the proof is complete.
\hfill
$\Box$

\section{Contraction Mapping and Homotopy}

The situation of a map $f\colon U \to V$ with $V \subset U$ with $f$
distance nonincreasing for some metric $d_U$ on $U$ is a natural 
condition for considering the ideas associated to iterations of 
\textit{contraction mappings}.  It is a familiar and long-standing principle
of analysis that in the case where $f$ is distance nonincreasing by a 
factor $0\le \alpha < 1$, then $f$ must have a fixed point if $U$ is 
complete with respect to the metric $d_U$.  
This is a natural way to prove the existence of a short-term solution of 
ordinary differential equations, the local surjectivity in the Inverse Function 
Theorem via Newton's Method, and many other basic results (Cf.\ e.g.\ 
\cite{Gamelin-Greene}).
Such contraction mapping ideas were used in \cite{Earle-Hamilton} to prove
a fixed point theorem for holomorphic maps in Banach spaces, now known as 
the Earle-Hamilton Fixed Point Theorem.  The finite 
dimensional version was proved earlier in \cite{Reiffen}.  In both papers, the 
Carath{\' e}odory metric rather than the Kobayashi metric was used.  We 
point out also that the finite dimensional version was proved even earlier in 
\cite{Herve} p. 83 (p. 92, in the 2nd ed.), using the fact that any compact analytic 
set in the complex Euclidean space consists of finitely many points.
On the other hand, the following theorem shows that the Kobayashi metric 
can also be used in the contraction mapping context in a way similar to 
\cite{Reiffen} and \cite{Earle-Hamilton}.

\begin{theorem} \label{fix-pt}
If $U$ is a bounded domain in $\CC^n$ and $f\colon U \to U$ is a 
holomorphic mapping such that $f(U)$ is contained in a compact subset of 
$U$, then there is a unique point $z_0 \in U$ such that $f(z_0)=z_0$.  
Moreover, this $z_0$ is exactly the only point such that the iterates 
$f^n (z) = f \circ \ldots \circ f (z)$, $n$-times, of $f$ converge uniformly 
on $U$ to the constant map at $z_0$ on $U$.
\end{theorem}

\textit{Proof}.
Choose an open set $V$ in $\CC^n$ with its compact closure $\overline{V}$
in $\CC^n$ satisfying 
\[
f(U) \subset V \subset \overline{V} \subset U.
\]
Let $d_U$ be the Kobayashi distance on $U$.  Lemma \ref{compare-2} shows
that there exists $0\le c < 1$ such that 
\[
F^\textrm{Kob}_U(f(p), df|_p (v)) \le c \cdot F^\textrm{Kob}_V (f(p),df|_p (v)),
~\forall (p,v) \in U \times \CC^n,
\]
which in turn implies that 
\begin{align*}
F^\textrm{Kob}_U (f\circ\gamma(t)),(f\circ\gamma)'(t)) 
&\le c \cdot F^\textrm{Kob}_V (f\circ\gamma(t)),(f\circ\gamma)'(t)) \\
&\le c \cdot F^\textrm{Kob}_U (\gamma(t), \gamma'(t))
\end{align*}
for any $C^1$ curve $\gamma\colon [0,1] \to U$.  So we have
\[
F^\textrm{Kob}_U (f\circ\gamma(t)),(f\circ\gamma)'(t)) 
\le c \cdot F^\textrm{Kob}_U (\gamma(t), \gamma'(t))
\]
for any $C^1$ curve $\gamma\colon [0,1] \to U$, and consequently
\[
d_U (f(p), f(q)) \le c \cdot d_U (p,q), ~\forall p, q \in U.
\]
Namely, $f$ is a strict contraction with the contraction factor $c$.

Now denote by $f^n$ the $n$-th iterate of $f$ defined inductively
by
\[
f^1 = f, ~f^{n+1} = f \circ f^n \quad (n=1,2,\cdots).
\]
If $p \in U$, then it follows that the sequence $f^n (p)$ is a 
Cauchy sequence with respect to $d_U$.  Even if $U$ were not 
necessarily $d_U$-complete, this Cauchy sequence still converges because, 
for $n \ge 1$, $f^n (p)$ belongs to $\overline{V}$, a
compact set with $\overline{V} \subset U$.  Thus 
$z_0 := \lim\limits_{n\to\infty} f^n (p)$ exists in $\overline{V}$, and 
consequently in $U$. And $f(z_0)=z_0$.  Of course this is the 
unique fixed point: if $f(z_1) = z_1$ then 
\[
d_U (z_0, z_1) = d_U (f(z_0), f(z_1)) \le c \cdot d_U (z_0, z_1),
\]
which implies $d_U (z_0, z_1) = 0$ and $z_0=z_1$.

The uniform convergence of $f^n$ on $U$ follows by 
\begin{align*}
d_U (z_0, f^{k+1} (p)) 
& = d_U (f^{k+1}(z_0), f^{k+1}(p)) \\
& \le c^k d_U (f(z_0), f(p)) \\
&\le c^k d_U (z_0, f(p)),
\end{align*}
which implies
\[
\sup_{p \in U} d_U (z_0, f^{k+1}(p)) 
\le c^k \sup_{f(p) \in V} d_U (z_0, f(p)).
\]
Note that $\sup_{f(p) \in V} d_U (z_0, f^m(p))$ is bounded, since the 
Kobayashi distance is a continuous function by \cite{Barth}. Moreover, there 
exists a positive integer $m$ such that for some constant $C>0$ it holds that 
\[
\frac1C\ d_U (z_0, f^n(p)) \le \| z_0 - f^n (p)\| \le C\  d_U (z_0, f^n(p))
\]
for any $n > m$.  Altogether, we see that the convergence of $f^k$ to 
the constant function at $z_0$ is uniform on $U$,
\hfill
$\Box$
\bigskip

This recovers Theorem \ref{homotopy_thm}: With the notation therein, if 
$f\colon T_{r_1} \to T_{r_2}$ with $r_1 > r_2$ is holomorphic then, since
$T_{r_2}$ is relatively compact in $T_{r_1}$, there is a point $z_0 \in T_{r_1}$
with $f(z_0)=z_0$.  Consider $f^n$ with $n$ large.  Then as mentioned before 
\[
\deg (f^n) = \big(\deg f\big)^n.
\]
But $f^n (T_{r_1})$ is contained in a small ball around $z_0$.  So the 
composition $P_{r_2} \circ f^n \circ i$ must have degree $0$, and the
rest of necessary arguments follows immediately.
\bigskip

A variant of this idea also produces this result about connected open
subsets in $\CC^n$ and their images under holomorphic maps to themselves
where the image is contained in a relatively compact subset.

\begin{theorem} \label{contr}
If $U$ is a bounded connected open set with smooth boundary in $\CC^n$,
and if there is a holomorphic map $f\colon U \to U$ such that $f(U)$ is 
contained in a compact subset of $U$ and that, for some $z \in U$ and 
for all $k=1,2,\ldots$, the induced map $f_*\colon \pi_k (U, z) 
\to \pi_k (U, f(z))$ is an isomorphism, then $U$ is contractible.
\end{theorem}

The theorem is illustrated by a bounded open sets $U$ with smooth
boundary that are star-shaped around the origin with $f(z) = rz$ for a 
constant $r$ with $0<r<1$. 
\medskip

The proof involves two lemmas:

\begin{lemma} \label{CW}
Every bounded open set $U$ in $\RR^N$ with smooth boundary has the 
homotopy type of a finite CW complex.
\end{lemma}

\begin{lemma} \label{trivial-homotopy}
With $U$ and $f$ as in the hypotheses of Theorem \ref{contr}, and 
with $z_0$ the fixed point of $f$ (which has already been shown to exist), 
the $k$-th homotopy group $\pi_k (U) = 0$ for any $k=1,2,\cdots$.
\end{lemma}

\textbf{The proof of Theorem \ref{contr}} follows from these two lemmas and the 
Whitehead theorem (\cite{Whitehead}. See also \cite{Hatcher}, 
Theorem 4.5 in page 346), since $f$ and the constant map 
$z_0$ have the same action on the homotopy groups of $U$ at $z_0$.
[The general result of Whitehead is that, 
\textit{if both $X$ and $Y$ have the homotopy type of connected finite
CW complexes and if two continuous maps $f, g \colon X \to Y$, 
where $f$ is a homotopy equivalence, satisfy the properties
that $f(x_0)=g(x_0)$ for some $x_0 \in X$ and $g_*, f_* :
\pi_k (X, x_0) \to \pi_k (Y,  f(x_0))$ are identical for every $k=1,2,\cdots$,
then $f$ is homotopic to $g$.}]
In our case, $f_* = 0$ and $g$ is to be the constant map at $z_0$, so both
$f_*$ and $g_*$ coincide as the trivial map of $\pi_k$ for every $k$.
\hfill
$\Box$
\bigskip

\begin{remark} \textup{
The open set $U$ is homotopically equivalent to a finite CW complex, but 
not homeomorphic to.  Note for instance that any finite CW complex is 
compact.  On the other hand, the fact that $U$ is of the same homotopy type 
with a finite CW complex is sufficient for the preceding proof.}
\end{remark}
\medskip

\textbf{Proof of Lemma \ref{CW}}: 
Let $d(z) = \textrm{dist } (z, \CC \setminus U)$, where 
$\textrm{dist } (z, A) = \inf \{\|z-w\| \colon w \in A\}$.  Here of
course $\|~\|$ is the Euclidean norm of $\CC^n$.  Since $U$ is open, 
$d(z) > 0$ for every $z \in U$ and the map $z \to d(z)$ is Lipschitz continous.
Also, since $U$ has smooth boundary, $d(z)$ is of class $C^1$ at least
($C^\infty$ if $U$ has $C^\infty$ boundary) on 
$\mathcal{N}_\epsilon := \{z \in U \colon d(z)<\epsilon\}$ 
for sufficiently small a constant $\epsilon > 0$, and 
$\|\textsl{grad } d\| = 1$ at every $z \in U$ with 
$d(z)<\epsilon$, where $\textsl{grad } d$ is the real Euclidean gradient.

Now set $\delta (z) = - \ln d(z)$.  Then $\delta$ is smooth on 
$\mathcal{N}_\epsilon$, and $\|\textsl{grad }\delta (z)\| \ge 1/\epsilon$
for every $z \in \mathcal{N}_\epsilon$.  (In fact, the inequality is $=$, but 
we do not need it here.) Now, for notational convenience,
let 
\[
K_1 = \{ z \in U \colon d(z) \ge \epsilon \}, \quad
K_m = \{ z \in U \colon d(z) \ge \epsilon/m \}
\]
for $m=1,2,\cdots$. (We shall need only the first few of these).  The sets $K_m$
defined as such are compact in $\CC^n$. 

Now $d$ may not be smooth on $U$; indeed, $d$ cannot be smooth since 
$d$ has a maximum in $U$ but $\|\textsl{grad } d\|=1$ at every point
at which $d$ is of class $C^1$.  So $\delta$ is also nonsmooth at some
points.  However, by standard convolution smoothing arguments, there is a 
smooth function, say $\Delta$, such that $\Delta$ is uniformly close to
$\delta$ on $K_4$ but (uniformly) $C^1$-close to $\delta$ on 
$K_4 \setminus K_1$.  By usual Morse theoretic considerations 
(\cite{Morse}.  Cf.\ Section I.6 of \cite{Milnor}. See also \cite{GW}.)
there is a function $\Delta_1$ which is uniformly $C^1$-close 
to $\Delta$ on $K_3$ and has only nondegenerate critical points.  Since 
$\Delta$ and $\delta$ are $C^1$-close on $K_4 \setminus K_1$ and 
$\|\textsl{grad } \delta\| \ge 1/\epsilon$ on $K_4 \setminus K_1$, it 
follows by the usual partition of unity argument that $\Delta$ on the interior
of $K_3$, and $\delta$ on $U$ can be patched together to yield a function
$\Delta_2$, say, such that 
\[
\Delta_2 (z) = \Delta (z), \forall z \in K_3
\]
with 
$\|\textsl{grad } \Delta_2\| \ge 1/(2\epsilon)$ on $U \setminus K_1$,
and such that 
\[
\Delta_2 (z) = \delta (z), \forall z \in U \setminus K_1.
\]
Then $\Delta_2$ is an exhaustion function for $U$ with only nondegenerate
critical points, since $\Delta_2$ has no critical points in $U \setminus K_1$ 
and hence only nondegenerate critical points in $U$, which necessarily 
lie in the set where $\Delta = \Delta_1$.

The standard Morse theory gives now that $U$ has the homotopy type of
a finite CW complex, with cells given by the finite number of nondegenerate
critical points of $\Delta_2$.
\hfill
$\Box$
\bigskip

\textbf{Proof of Lemma \ref{trivial-homotopy}}:
Let $z_0$ be the fixed point of $f$.  Suppose that 
$\Gamma\colon S^k \to U$ is a representation of a $k$-homotopy
class in $\pi_k (U, z_0)$.  By Theorem \ref{fix-pt},  
the iterates $f^n \circ \Gamma$ converge
uniformly on $S^k$ to the constant map at $z_0$.  In particular,
if a positive constant $r$ is such that 
$B^n (z_0, r) = \{z \colon \|z-z_0\|<r\} \subset U$
and $f^n \Gamma (S^k) \subset B^n (z_0, r)$, then 
$[f^n\circ \Gamma] = 0$ in $\pi_k (U,z_0)$. But by hypothesis, 
$f^n_* \colon \pi_k (U, z_0) \to \pi_k (U, z_0)$ is an isomorphism.  
So $[\Gamma] = 0$ in $\pi_k (U, z_0)$.
\hfill
$\Box$
\bigskip

\section*{Acknowledgements}

After this article was written and posted in \texttt{arXiv.org}, 
T. Pacini kindly informed us the relevance of \cite{Pacini}.  We express our
gratitude for this. 

Research of the second named author (Kim) is partially supported by
the NRF Grant 4.0021348 of The Republic of Korea.


\begin{thebibliography}{99}

\bibitem{Ahlfors} L. V. Ahlfors: Complex analysis (3rd ed.), \textit{McGraw-Hill}. 
1979.

\bibitem{Barth} T. J. Barth: \textit{The Kobayashi distance induces the standard topology}. Proc. Amer. Math. Soc. 35 (1972), 439–441.

\bibitem{dePossel}
R. de Possel; \textit{Sur quelques propri{\' e}t{\' e}s de la repr{\' e}sentation 
conforme des domaines multiplement connexes, en relation avec le 
th{\' e}orème des fentes parallèles}. Math. Ann. 107 (1933), no. 1, 496–504.

\bibitem{Earle-Hamilton}
C. J. Earle; R. S. Hamilton: A fixed point theorem for holomorphic mappings. 
1970 Global Analysis; Proc. Sympos. Pure Math., Vol. XVI, Berkeley, Calif., (1968) 
pp. 61–65 Amer. Math. Soc.

\bibitem{Gamelin-Greene} 
T. Gamelin; R. E. Greene: Introduction to topology. Second edition. \textit{Dover Publications, Inc.}, Mineola, NY, 1999. xiv+234 pp.

\bibitem{Goluzin} G. M. Goluzin: Geometric theory of functions of a 
complex variable, \textit{AMS translation series}, 1969.

\bibitem{Graham} I. Graham: \textit{Boundary behavior of the Carath{\' e}odory, 
Kobayashi, and Bergman metrics on strongly pseudoconvex domains in 
$\CC^n$ with smooth boundary}. Trans. A. M. S. 207 (1975), 219-240.

\bibitem{GKK} R. E. Greene; K.-T. Kim; S. G. Krantz: The geometry
of complex domains, Progr. Math. 291\textit{Birkhäuser}, 2011.

\bibitem{GK} R. E. Greene; S. G. Krantz: Function theory of one 
complex variable (3rd ed.), \textit{Amer.\ Math.\ Soc.} 2006.

\bibitem{GW} R. E. Greene; H. Wu: $\mathcal{C}^\infty$ 
approximations of convex, subharmonic, and plurisubharmonic function
Ann. Sci. l'{\' e}cole Norm, Sup., S{\' e}rie 4, Tome 12 (1979) no. 1, pp. 47-84.


\bibitem{Grotzsch} H. Gr{\" o}tzsch; \textit{Zur Theorie der Verschiebung bei 
schlichter konformer Abbildung}. Comment. Math. Helv. 8 (1935), no. 1, 
382–390.

\bibitem{Hatcher} A. Hatcher: Algebraic Topology (2001),  
\textit{Cambridge University Press}, Cambridge, 2002. xii+544 pp.

\bibitem{Herve} M. Herv\'e: Several complex variables, local theory. \textit{Oxford U. Press and Tata Inst. Fund. Research}. 1963.

\bibitem{KimLee} K.-T. Kim; H. Lee: Schwarz's lemma from a differential
geometric viewpoint. \textit{World Scientific}, 2011.

\bibitem{Kobayashi70} S. Kobayashi: Hyperbolic manifolds and holomorphic
mappings, \textit{Marcel-Dekker}, 1970.

\bibitem{Kobayashi98} S. Kobayashi: Hyperbolic complex spaces,
\textit{Springer} 1998.

\bibitem{Littlewood} J. E. Littlewood: 
\textit{Quelques consequences de l'hypothese que la fonction $\zeta(s)$ de 
Riemann n'a pas de zeros dans le demi-plan $\re (s) > 1/2$}, 
Les Comptes rendus de l'Acad\'emie des sciences, 154 (1912), 263--266.

\bibitem{Milnor}  J. Milnor: Morse theory. 
Based on lecture notes by M. Spivak and R. Wells, 
\textit{Annals of Mathematics Studies}, 
No. 51 Princeton University Press, Princeton, N.J. 1963.

\bibitem{Morse} M. Morse: The critical points of a function of  variables,
Trans. Amer. Math. Soc. 33 (1931), 72-91.

\bibitem{Pacini} T. Pacini: \textit{Extremal length in higher dimensions and complex systolic inequalities}. J. Geom. Anal. 31 (2021), no. 5, 5073–5093.

\bibitem{Petersen} P. Petersen: Riemannian goemetry (2nd ed.), \textit{
Grad. Text. Math. 171,
Springer}, 2006.

\bibitem{Reiffen} H.-J. Reiffen: Die Carath\'eodorysche Distanz und ihre zugeh\"orige Differentialmetrik. \textit{Math. Ann.} 161 (1965), 315-324.

\bibitem{Royden70} H. L. Royden: \textit{Report on the Teichmüller metric}. Proc. Nat. Acad. Sci. U.S.A. 65 (1970), 497–499.

\bibitem{Royden71} H. L. Royden: \textit{Remarks on the Kobayashi metric},
Several complex Variables, II. Maryland 1970. Springer 1971.

\bibitem{Rudin} W. Rudin: Real and complex analysis, 
\textit{McGraw Hill}, 1987. 

\bibitem{Whitehead} J. H. C. Whitehead: \textit{Combinatorial homotopy II}. Bull. 
A.M.S. 55 (1949), 453–496.


\end{thebibliography}
\end{document}